\documentclass[12 pt]{article}
\usepackage{amsmath}
\usepackage{amsfonts}
\usepackage{amssymb}
\usepackage{amsthm}
\usepackage{thmtools}
\usepackage{hyperref}
\usepackage{mathtools}
\usepackage{enumitem}
\usepackage{graphicx}
\usepackage[style=numeric,backend=bibtex]{biblatex}

\newtheorem{thm}{Theorem}[section]
\newtheorem{lem}[thm]{Lemma}

\allowdisplaybreaks

\title{3-braid knots do not admit purely cosmetic surgeries}
\author{Konstantinos Varvarezos}
\begin{document}
\maketitle

\begin{abstract}
A pair of surgeries on a knot is called purely cosmetic if the pair of resulting 3-manifolds are homeomorphic as oriented manifolds.  Using recent work of Hanselman, we show that (nontrivial) knots which arise as the closure of a 3-stranded braid do not admit any purely cosmetic surgeries.
\end{abstract}


\section{Introduction}

Given a knot $K$ in $S^3$ and a pair of coprime integers $p,q$, we denote the \textit{Dehn surgery} on $K$ with slope $p/q$ by $S^3_{p/q}(K).$  Surgeries on $K$ along distinct slopes $r$ and $r'$ are called \textit{cosmetic} if $S^3_r(K)$ and $S^3_{r'}(K)$ are homeomorphic manifolds.  Furthermore, a pair of such surgeries is said to be \textit{purely cosmetic} if $S^3_r(K)$ and $S^3_{r'}(K)$ are homeomorphic as \textit{oriented} manifolds, whereas if they are homeomorphic but have opposite orientations, the pair of surgeries is called \textit{chirally cosmetic}.

No purely cosmetic surgeries are have been found on nontrivial knots in $S^3$; indeed, the \textit{Cosmetic Surgery Conjecture} predicts that none exist (compare also Problem 1.81(A) in \cite{KirbyList}).  On the other hand, there are examples of chirally cosmetic surgeries.  For instance, $S^3_r(K) \cong -S^3_{-r}(K)$ whenever $K$ is an amphicheiral knot.  Also, $(2,n)$-torus knots are known to admit chirally cosmetic surgeries; see \cite{IIS,Mat}.

In this work, we verify the Cosmetic Surgery Conjecture for 3-braid knots.
\begin{restatable}{thm}{main}\label{thm:main}
Suppose $K$ is a nontrivial knot which is the closure of a 3-braid.  Then $K$ admits no purely cosmetic surgeries.
\end{restatable}

\subsection*{Acknowledgements}
The author would like to thank Professor Zolt\'{a}n Szab\'{o}  for suggesting work on this problem as well as many helpful conversations.  The author also thanks Professor Kazuhiro Ichihara for useful comments and corrections.  This work was supported by the NSF RTG grant DMS-1502424.

\section{Preliminaries}
\subsection{Band presentation and genus}
Here we recall some useful results about the genera of 3-braid knots.  We shall make use of the so-called ``band presentation" of the 3-braid group, studied for braid groups in general by Birman, Ko, Lee \cite{BKL} and which has proved useful in the study of surfaces bounded by braid knots; see, for instance, \cite{Rud,Xu,LL}.  For 3-braids, the presentation is as follows:
\[
\left\langle a_1,a_2,a_3 | a_2 a_1 = a_3 a_2 = a_1 a_3 \right\rangle
\]
Here $a_1$ and $a_2$ correspond to the standard Artin generators $\sigma_1$ and $\sigma_2$ respectively, and $a_3$ corresponds to  $\sigma_2 \sigma_1 \sigma_2^{-1}$; see Figure \ref{fig:bgens} for an illustration.

\begin{figure}
\centering
\includegraphics[width=.75\textwidth]{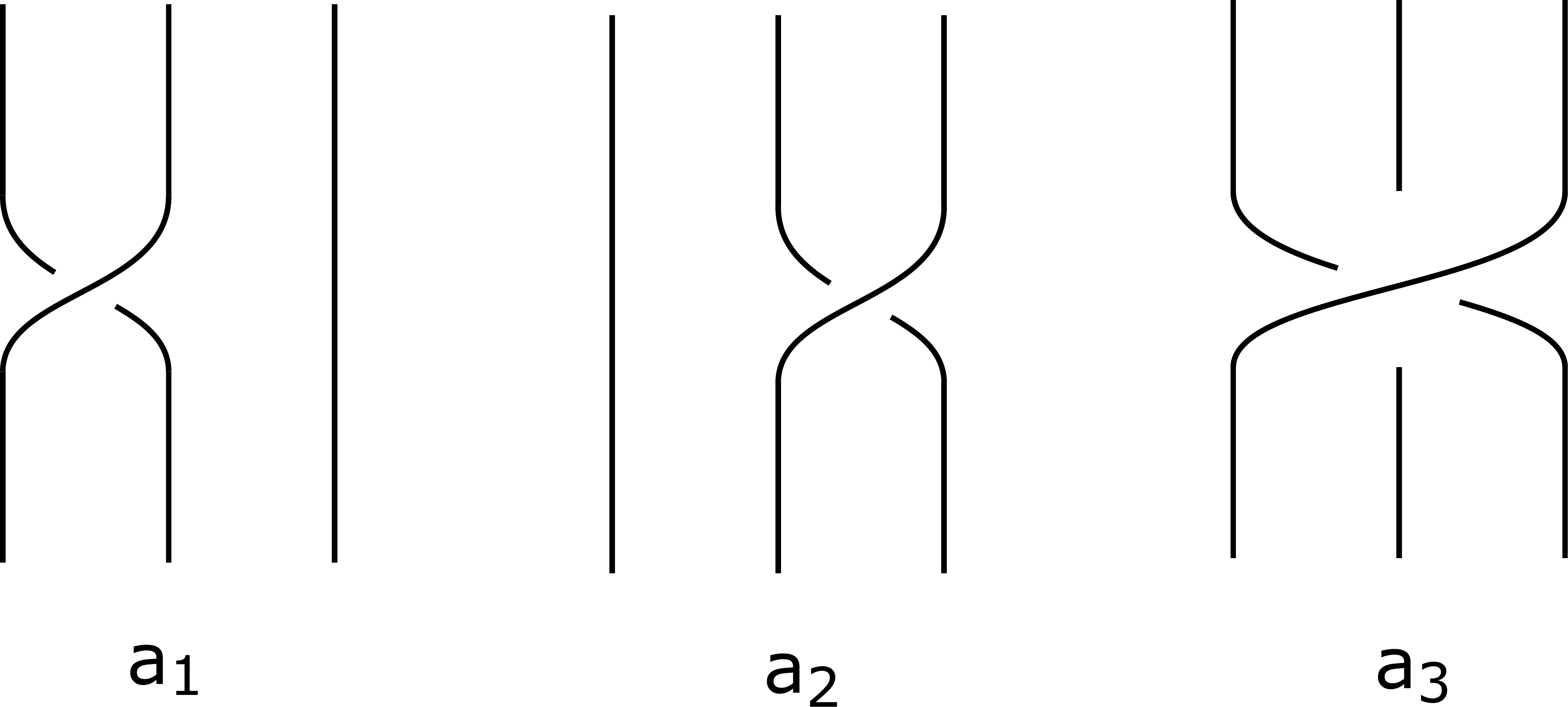}
\caption{The three generators of the band presentation of the 3-braid group.}	
\label{fig:bgens}
\end{figure}

A word in the band generators very naturally gives rise to a Seifert surface for the corresponding closed braid, which we shall call the \textit{banded surface} associated to the word.  In particular, starting with three disks, which one can imagine as being vertically ``stacked", one attaches a half-twisted band between two of the disks for each instance of $a_j^{\pm 1}$: for $a_1$ between the bottom and middle disks, for $a_2$ between the middle and top, and for $a_3$ between the top and bottom, twisted appropriately depending on the sign of the exponent; see Figure \ref{fig:surf} for the banded surface corresponding to the braid $a_2 a_3^{-1} a_1^2$.
\begin{figure}
\centering
\includegraphics[width=.5\textwidth]{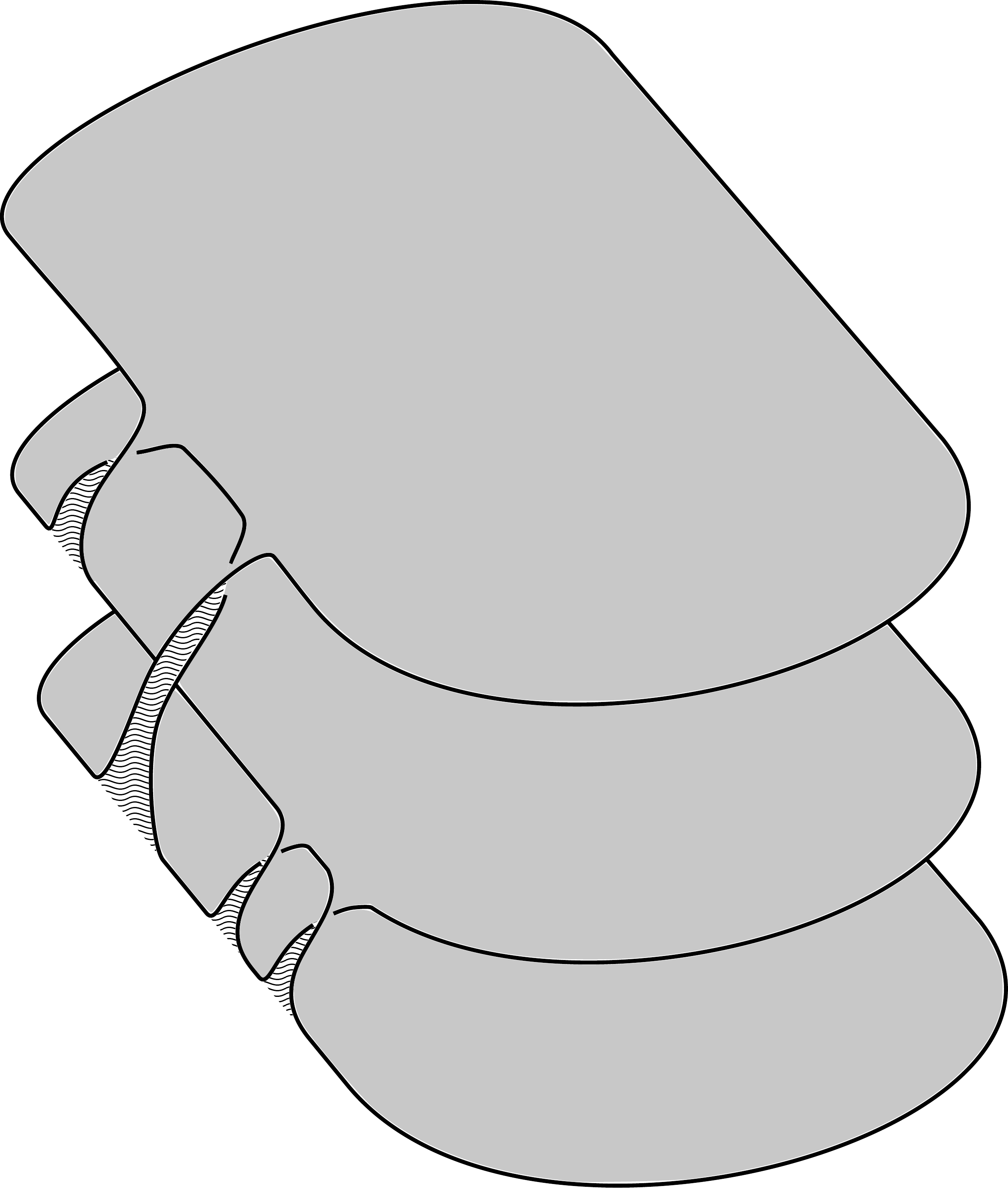}
\caption{A banded surface for the 3-braid knot corresponding to the closure of the braid $a_2 a_3^{-1} a_1^2$.}	
\label{fig:surf}
\end{figure}

Now suppose that a 3-braid in the band generators has a knot as its braid-closure.  A straightforward computation of the Euler characteristic reveals that the banded surface produced with the above method has genus
\begin{equation}\label{eq:gen}
g = \frac{\ell - 2}{2}
\end{equation}
where $\ell$ is the length of the braid word (that is, the number of instances of the $a_j^{\pm 1}$s).  Moreover, in the case of 3-braids, a theorem of Bennequin's says that a minimal Seifert surface can always be produced in this way (see also \cite{BM2}):
\begin{thm}[Proposition 3 of \cite{Ben}]
If $K$ is a the closure of a 3-braid, then the Seifert genus of K can be realized by a banded surface.
\end{thm}

Suppose $K$ is a knot which is the closure of a 3-braid expressed as a word in the band generators.  By Bennequin's theorem, we may take this word to be minimal so that the associated banded surface achieves the Seifert genus of the knot $g(K)$.  Notice that each instance of $a_1^{\pm 1}$ and $a_2^{\pm 1}$ contributes a crossing to the diagram of the braid, while each instance of $a_3^{\pm 1}$ contributes 3 crossings.  Let $A_j$ denote the total number of instances of $a_j^{\pm 1}$ in the given word for $K$.  Then $K$ admits a diagram with $A_1 + A_2 + 3A_3$ crossings.  Further, observe that $K$ is isotopic to the closure of the braid obtained by cyclically permuting the generators: $a_1 \rightarrow a_2 \rightarrow a_3 \rightarrow a_1$ (indeed, one can even realize this as an isotopy of a banded surface like the one in Figure \ref{fig:surf} by ``flipping" the top disk across the banded areas so that it becomes the bottom disk).  Thus, we may assume that $A_3$ is the least of the $A_j$s; that is, $A_3 \leq \frac{1}{3}(A_1+A_2+A_3)$.  On the other hand, we have that the length of the word $\ell = A_1+A_2+A_3$.  Combining this with equation \eqref{eq:gen}, we see that $K$ admits a diagram with crossing number:
\begin{align*}
c &= A_1 + A_2 + 3A_3 \\
&= \ell + 2A_3 \\
&\leq \ell +\frac{2}{3}\ell \\
&=\frac{10}{3}(g(K)+1).
\end{align*}
Therefore, we have shown the following lemma:
\begin{lem}\label{lem:gen}
Let $K$ be a 3-braid knot. If $g(K)$ denotes the Seifert genus of $K$, then $K$ admits a diagram with no more than $\frac{10}{3}(g(K)+1)$ crossings.
\end{lem}
\begin{proof}[Remark.]\let \qed \relax
In fact, we may additionally require that $K$ admit a diagram with an \textit{even} number of crossings less than or equal to $\frac{10}{3}(g(K)+1)$.  This follows from the fact that, in order for the closure of a 3-braid to result in a knot, the permutation of the strands induced by the braid must be even (indeed it must be one of the cycles of order 3 in $S_3$, the symmetric group on the three strands).
\end{proof}
\subsection{Knot Floer Homology and thickness}
Knot Floer Homology is a knot invariant introduced by Ozsv\'{a}th and Szab\'{o} \cite{OSzHFK} and also independently by Rasmussen \cite{Ras}.  It is related to Heegaard Floer Homology, a 3-manifold invariant also developed by Ozsv\'{a}th and Szab\'{o}, and it has been used with much success to obtain general obstructions for cosmetic surgeries; see, for instance, \cite{Wang,OSz,NW}.  In this work, we shall make use of obstructions coming from recent work of Hanselman \cite{Hans}.

Given a knot $K$, we consider the so-called ``hat" version of Knot Floer Homology, denoted $\widehat{HFK}(K)$, which is the homology of a bigraded chain complex denoted $\widehat{CFK}(K)$.  The two gradings are called the Alexander and Maslow gradings.  In an algebraic re-interpretation of Knot Floer Homology, Ozsv\'{a}th and Szab\'{o} showed \cite{OSzBord,OSzAlg} that the generators of $\widehat{CFK}(K)$ may be identified with so-called Kauffman states of a diagram for $K$.  A Kauffman state consists of a choice of one of the four ``corners" for each crossing appearing in the diagram for $K$, such that all but two (chosen once and for all) of the ``regions" that make up the complement of the diagram have exactly one marking (the two ``distinguished" regions never receive any markings).  We denote the set of all Kauffman states by $\mathfrak{S}$ (note that this depends on the choice of diagram for $K$).  The Alexander and Maslow grading of each such generator is the sum of the local grading contributions at each crossing, as shown in Figure \ref{fig:grad}.  The \textit{delta} grading of a generator of $\widehat{CFK}(K)$ is defined to be the difference between its Alexander and Maslow gradings; hence, it too may be computed as a sum of local delta-gradings, also displayed in Figure \ref{fig:grad}.

\begin{figure}
\centering
\includegraphics[width=.75\textwidth]{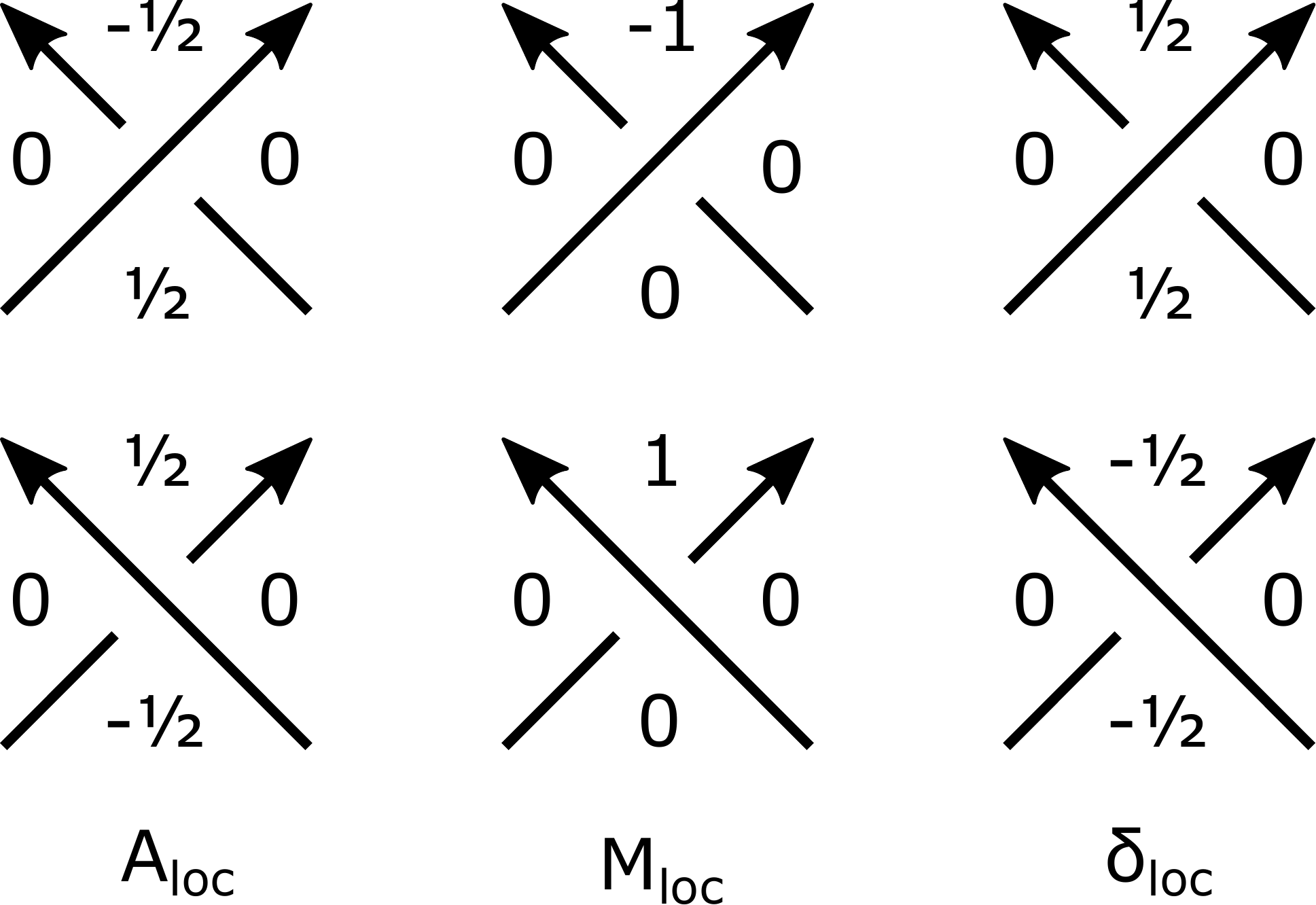}
\caption{The local Alexander, Maslow, and delta grading contributions at a crossing.  The top row shows a positive crossing and the bottom shows a negative crossing.}	
\label{fig:grad}
\end{figure}
The (homological) \textit{thickness} of a knot $K$, which we shall denote $th(K)$, is defined to be the difference between the maximal and minimal delta gradings of the support of $\widehat{HFK}(K)$.  Using the Kauffaman state interpretation, we obtain the following bound on the thickness:
\begin{align*}
th(K) &\leq \max_{s \in \mathfrak{S}} \delta(s) - \min_{s \in \mathfrak{S}} \delta(s) \\
&= \max_{s \in \mathfrak{S}} \sum_{c \in s} \delta_{\mathrm{loc}}(c) - \min_{s \in \mathfrak{S}} \sum_{c \in s} \delta_{\mathrm{loc}}(c) \\
&\leq \sum_{\substack{c \in s \\ \text{positive}}} \frac{1}{2} -  \sum_{\substack{c \in s \\ \text{negative}}} -\frac{1}{2} \\
&= \frac{1}{2}(n^+ + n^-)
\end{align*}
where a sum over $c \in s$ means summing over all the (marked) crossings $c$ of the Kauffman state $s$, and where $n^+$ and $n^-$ denote the number of positive and negative crossings appearing in the diagram, respectively.  Hence, the thickness of a knot $K$ is at most one-half the number of crossings appearing in any diagram for $K$.  Combining this with Lemma \ref{lem:gen}, we obtain:
\begin{lem}\label{lem:gth}
If $K$ is a knot which is the closure of a 3-braid, then
\[
th(K) \leq \frac{5}{3}(g(K)+1)
\]
\end{lem}

\section{Proof of Theorem \ref{thm:main}}
We are now ready to prove our main result.  We shall make use of the following obstruction obtained by Hanselman for the existence of purely cosmetic surgeries:
\begin{thm}[Theorem 2 of \cite{Hans}]\label{thm:Hans}
Let $K \subset S^3$ be a nontrivial knot and suppose that $S_r^3$ and $S^3_{r'}$ are homeomorphic as oriented manifolds for $r,r'$ distinct rational numbers.  Then $\lbrace r,r'\rbrace = \lbrace \pm 2\rbrace \text{ or } \lbrace \pm \frac{1}{q} \rbrace$ for some positive integer $q$.  Moreover:
\begin{enumerate}[label=\roman*)]
\item If $\lbrace r,r'\rbrace = \lbrace \pm 2\rbrace$ then $g(K)=2$
\item If $\lbrace r,r'\rbrace = \lbrace \pm \frac{1}{q} \rbrace$ then $\displaystyle q\leq \frac{th(K)+2g(K)}{2g(K)(g(K)-1)}$ \label{ineq}
\end{enumerate}
\end{thm}
\begin{proof}[Remark.]\let \qed \relax Notice that if for a knot $K$ with genus greater than 2 the fraction on the right hand side of the inequality in \ref{ineq} is less than one, then $K$ admits no purely comsetic surgeries.
\end{proof}

\main*
\begin{proof}
First, consider a 3-braid knot $K$ with genus $g(K)=g \geq 4$.  By Lemma \ref{lem:gth}, $th(K)\leq \frac{5}{3}(g+1)$.  Hence,
\begin{align*}
\frac{th(K)+2g(K)}{2g(K)(g(K)-1)} &\leq \frac{\frac{11}{3}g +\frac{5}{3}}{2g(g-1)} \\
&= \frac{1}{g-1}\left(\frac{11}{6}+\frac{5}{6g}\right) \\
&\leq \frac{1}{3}\left(\frac{11}{6}+\frac{5}{24}\right) = \frac{49}{72} < 1
\end{align*}
where the hypothesis that $K$ has genus at least 4 was used in the last step.  By the remark following Theorem \ref{thm:Hans}, $K$ admits no purely cosmetic surgeries.

So we are left with the case $g(K) \leq 3$.  By Lemma \ref{lem:gen} along with the remark following it, any such knot has crossing number at most 12.   As composite knots were shown not to admit any purely cosmetic surgeries by Tao \cite{Tao}, we may restrict our attention to prime knots.  On the other hand, if $g(K) \leq 2$, it has crossing number at most 10.  In \cite{Ito}, Ito showed that the Cosmetic Surgery Conjecture holds for all prime knots of crossing number at most 11, with the possible exception of $10_118$.  According to KnotInfo \cite{KI}, this knot is a 3-braid knot of genus 4, and so admits no purely cosmetic surgeries by the above computation.  Hence, the remaining unchecked knots have genus 3 and crossing number 12.  A search through KnotInfo reveals three knots satisfying these conditions (with bridge index 3).  KnotInfo also gives the Alexander polynomials of these knots, from which we may compute the invariant $a_2(K) = \frac{1}{2}\Delta_K''(1)$, where $\Delta_K(t)$ denotes the symmetrized Alexander polynomial of $K$.  The results are shown in the following table:

\begin{center}
\begin{tabular}{|c|c|c|}
\hline
$K$ & $\Delta_K(t)$ & $a_2(K)$ \\
\hline
$12n749$ & $t^{-3} - t^{-2}+t^{-1}-1+t-t^2+t^3$ & 6 \\
\hline
$12n750$ & $2 t^{-2}-5t^{-1}+7-5t+2t^2$ & 3 \\
\hline
$12n830$ & $2 t^{-2}-4t^{-1}+5-4t+2t^2$ & 4 \\
\hline
\end{tabular}
\end{center}
By a result of Boyer and Lines \cite{BL}, if a knot $K$ admits purely cosmetic surgeries then $a(K)=0$.  Hence the three remaining knots are ruled out as well.
\end{proof}

\printbibliography

\end{document}